\documentclass[ 12pt]{amsart}
\usepackage{amsfonts}
\usepackage{mathrsfs}
\usepackage{bbm}
\usepackage{amsfonts}
\usepackage{mathrsfs}
\usepackage{bbm}
\usepackage{amsfonts,amsthm}
\usepackage{amssymb,amsmath,graphicx}
\usepackage{cases}
\usepackage{bm}
\usepackage{float}
\usepackage[colorlinks=true]{hyperref}
\hypersetup{urlcolor=blue, citecolor=red}
\usepackage{hyperref}
\usepackage{enumerate}
\usepackage{cite}
\usepackage{amsfonts,amsthm}
\usepackage{amssymb,amsmath,graphicx}
\usepackage{cases}
\usepackage{float}
\usepackage[colorlinks=true]{hyperref}
\hypersetup{urlcolor=blue, citecolor=red}
\usepackage{hyperref}
\usepackage{enumerate}
\usepackage{cite}
\newtheorem{theorem}{Theorem}

\newtheorem{lemma}{Lemma}

\newtheorem{remark}{Remark}

\def\XXint#1#2#3{{\setbox0=\hbox{$#1{#2#3}{\int}$ }
		\vcenter{\hbox{$#2#3$ }}\kern-.6\wd0}}

\def\qaq{\quad\text{and}\quad}
\def\Div{{\rm div}}
\def\bx{{\mathbf{x}}}
\def\by{{\mathbf{y}}}
\def\bz{{\mathbf{z}}}
\def\bu{{\mathbf{u}}}
\def\bn{{\mathbf{n}}}

\everymath{\displaystyle}
\begin{document}
	\title[Convergence Rate]{On the convergence rates of multi-dimensional subsonic irrotational flows in unbounded domains}
\author{Lei Ma}
\address{Lei Ma, College of Science, University of Shanghai for Science and Technology, 516 Jungong Road, Shanghai, 200093, China}
\email{leima@usst.edu.cn}
\author{Tian-Yi Wang}
\address{Tian-Yi Wang, Department of Mathematics, School of Science, Wuhan University of Technology, 122 Luoshi Road, Wuhan, Hubei, 430070, China}
\email{tianyiwang@whut.edu.cn}

\date{}
	\begin{abstract}	This paper is concerned with the convergence rates of subsonic flows for airfoil problem and infinite long largely-open nozzle problem, which is an improvement of \cite{MR1211737, MR0092912, MR3148105, MR92913}. The maximum principle is applied to estimate the potential function, by choosing the proper compared functions. Then, by the weighted Schauder estimates, the convergence rates of velocity at the far field are shown as $|\bx|^{-n+1}$. Furthermore, we construct the examples to show the optimality of our convergence rates and the expansion of the incompressible airfoil flow at infinity, indicating the higher convergence rates  $|\bx|^{-n}$.\\[3mm]{\bf Keywords:}    Convergence rates, Subsonic irrotational flows, Airfoil problem, Infinite long largely-open nozzle
		\\[3mm]{\bf Mathematics Subject Classification 2020:}  35J15 
		35B40 
		76G25 
		35Q35 
		76Nxx 
	
\end{abstract}
\maketitle
\section{Introduction}

We are concerned with  the convergence rates of multi-dimensional subsonic irrotational flows in unbounded domains. The $n$-dimensional$(n\geq3)$ steady homentropic Euler equations  are written as:
	\begin{equation}\label{eulereq-1}
	\begin{cases}
		\Div(\rho  {\bf u}) = 0,   \\
		\Div(\rho {\bf u}\otimes {\bf u})+\nabla p=0,
	\end{cases}
\end{equation}
where $\bu=(u_1,\cdots,u_n)$,  $\rho$, and $p$ represent the velocity, density, and pressure respectively.
For the homentropic flow, $p$ is a function of $\rho$, which satisfies $C([0, +\infty))\cap C^2((0, +\infty))$ with
\begin{equation}\label{conditiononpressure}
	p'(\rho)>0,  \quad 2p'(\rho) + \rho p''(\rho) > 0 \qquad \mbox{for $\rho>0$},
\end{equation}
which holds for the flows governed by the thermodynamic relation that $p=\frac{1}{\gamma} \rho^\gamma$ with $\gamma\geq 1$. Suppose that the flow is irrotational, \textit{i.e.},
 \begin{equation}\label{irrotational}
\text{curl}\  {\bf u}=0.
 \end{equation}
 Combining \eqref{eulereq-1} and \eqref{irrotational}, we could consider the following equivalence form:
 \begin{equation}\label{eulereq}
 	\begin{cases}
 		\Div(\rho  {\bf u}) = 0,   \\
 		{\rm curl}\  {\bf u}=0,\\
 		\frac{1}{2}|\bu|^2+h(\rho)\equiv B,
 	\end{cases}
 \end{equation}
 while the last equation is called the Bernoulli's law. Here, $B$ is a constant, and $h(\rho)$ is the enthalpy defined by $h^\prime(\rho)=\frac{p^\prime(\rho)}{\rho}$. In virtue of \eqref{conditiononpressure} and  $(\ref{eulereq})_3$, the density $\rho$ could be regarded as the function of speed $|\bu|$, which could be written as $\rho(|\bu|^2)$.

 By the sound speed $ c(\rho)=\sqrt{p'(\rho)}$,
 the flow could be classified as  supersonic, sonic, or subsonic when the $|{\bf u}|>c$, $=c$, $<c$, respectively. Furthermore, there is a critical speed $q_c$ such that the speed $|\bu|<q_{cr}$ if and only if $|\bu|<c(\rho)$\cite{1948sfswbookC}. For $p(\rho)=\frac{1}{\gamma}\rho^\gamma$ with $\gamma\geq 1$, the critical speed is $q_{cr}=\sqrt{\frac{2}{\gamma+1}}$.

The main objective of this paper is to investigate the convergence rates of velocity $\bu$ at far fields for the following two types of problems in the unbounded domains.

For Airfoil problem, let $\mathcal{B}(\Gamma)$ (airfoil) include finite disjoint bounded branch in $\mathbb{R}^n$ ($n\geq3$) such that its boundary $\Gamma$ consists of one or several closed and isolated $n-1$ dimensional $C^{2,\alpha}$ (for some $0<\alpha<1$) hypersurfaces. Let $\Omega^\mathcal{B}$ be the exterior domain
of $\mathcal{B}(\Gamma)$, \textit{i.e.}, $\Omega^\mathcal{B}:=\mathbb{R}^n\backslash\mathcal{B}$, which is connected, see Figure \ref{fig:1}. Then, we could introduce the following problem:

\begin{figure}
	\begin{center}
		\includegraphics[]{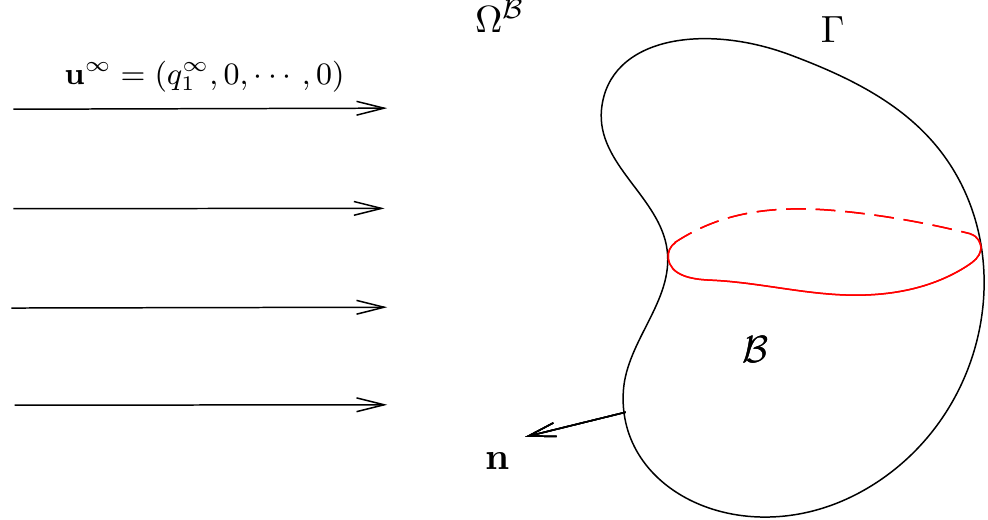}\\
	\centering	\caption{Exterior domain}
	\end{center}\label{fig:1}
\end{figure}

\textbf{Airfoil Problem}.  Find functions $\bu$ satisfy \eqref{eulereq} in $\Omega^\mathcal{B}$
with the slip boundary
condition
\begin{equation}\label{boundarycondition1}
	\bu\cdot \bn=0\ \ \mbox{on}\ \ \Gamma,
\end{equation}
where $\bn=(n_1,\cdots,n_n)$ denotes the unit outer normal of domain
$\Gamma$. Write $\bx=(x_1, \cdots, x_n)$, the flow satisfies the following asymptotic condition
\begin{equation}\label{u-inftycondtion}
	\lim_{|\bx|\rightarrow \infty} \bu(\bx)=\bu^\infty.
\end{equation}
Without loss of the generality, we also assume that $\bu^{\infty}=(q^{\infty}_1,0, \cdots,0)$.

To consider the convergence rates at the far field, we could introduce the respective far field flow as:

\textbf{Airfoil Far Field Flow}.  $\bu$ satisfy \eqref{eulereq} and \eqref{u-inftycondtion} in $\Omega^{\mathcal B}\cap\{|\bx|\geq R_1\}$ with $R_1>0$ be a sufficiently large number.

It is easy to see the solution $\bu$ of \textbf{Airfoil Problem} is \textbf{Airfoil Far Field Flow}, but may not other way around, due to the different boundaries and the boundary conditions.

The other type of the unbounded domain is the infinite long largely-open nozzle  $\Omega^\mathcal{N}\subset\mathbb{R}^n$(see Figure \ref{NozzleF}) which is connected by divergent domain $\mathcal C_+$, convergent part $\mathcal C_-$, and the throat part $\mathcal{T}$. Here $\mathcal{C}_+$ (resp. $\mathcal{C}_-$) is an infinite long part of a cone with the cross section $\Sigma_+$(resp. $\Sigma_-$) whose center is  $A_+$ (resp. $A_-$), where $\Sigma_+$ (resp. $\Sigma_-$) is a connected smooth part of the unit ($n-1$)-sphere $\mathbb{S}^{n-1}\subset\mathbb{R}^n$. $\mathcal T$ is a bounded domain connected $\mathcal{C}_+$ and $\mathcal{C}_-$ such that $\partial\Omega^{\mathcal N}$ is $C^{2,\alpha}$ with $0<\alpha<1$. 

\begin{figure}
	\begin{center}
		\includegraphics[]{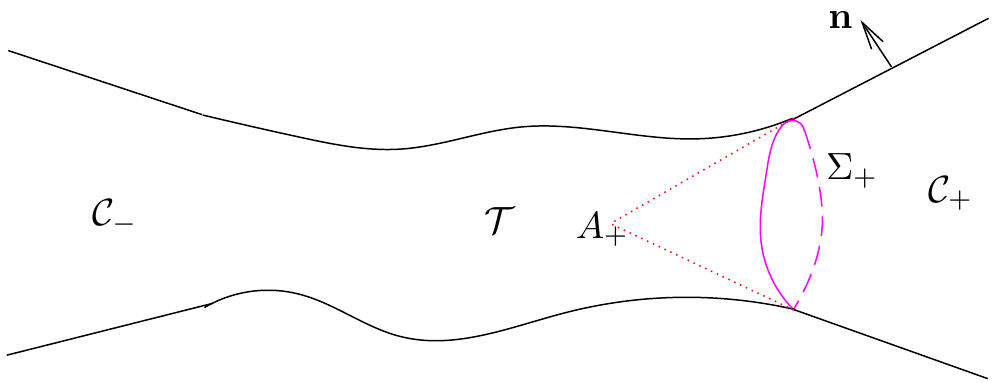}\\
		\caption{Nozzle domain}\label{NozzleF}
	\end{center}
\end{figure}

The fluid fills in the region $\Omega^\mathcal{N}$ and  satisfies the slip boundary condition:
\begin{eqnarray}\label{boundarycondition2}
	\bu\cdot\bn=0\quad \mbox{on} \quad \partial{\Omega^\mathcal{N}},
\end{eqnarray}
where $\bn$ is the unit outward normal of domain $\partial\Omega^\mathcal{N}$. Applying the divergence theorem to $(\ref{eulereq})_1$ and \eqref{boundarycondition2}, one can obtain the fixed mass flux property:
\begin{equation}\label{flux}
	\int_{\mathcal{S}_0}\rho \bu\cdot \mathbf{l}ds=m,
\end{equation}
where $\mathcal{S}_0$ is any cross-section of $\Omega^\mathcal{N}$ and $\mathbf{l}$ is the unit outer vector of $\mathcal S_0$ pointed into downstream. $m>0$ is called the mass flux.
Then, we could introduce the following problem:

\textbf{Largely-open Nozzle Problem}. Find functions $\bu$, which satisfies \eqref{eulereq}, with the slip boundary condition \eqref{boundarycondition2} and mass flux condition \eqref{flux}.

The respective far field flow is:

\textbf{Largely-open Nozzle Far Field Flow}. $\bu$ satisfies \eqref{eulereq} in $\Omega^{\mathcal N}\cap\{|\bx|\geq R_1\}$ with the slip boundary condition \eqref{boundarycondition2} and mass flux condition \eqref{flux}.

The study of subsonic flow has a long research history \cite{Bers1958Mathematical}. The first mathematical result is on  the existence and uniqueness of \textbf{Airfoil Problem} in 2 dimension for small data by \cite{Frankl1934Keldysh}.  An important progress was made by Shiffman \cite{shiffman1952existence} via the variational method. Later on, Bers \cite{MR0065334} proved the existence of two dimensional irrotational subsonic flows around a profile with a sharp trailing edge. The uniqueness theory of the subsonic plane flow was established in \cite{MR0086556}. For the three (or higher) dimensional case, Dong and Ou introduced a proper Hilbert space
for the variational method  and obtained the existence and uniqueness of \textbf{Airfoil Problem} in \cite{MR1211737}, please also see \cite{MR1134129}. The  well-posedness of two and higher dimensional subsonic irrotational flows in the infinite long nozzle had been investigated in \cite{MR2375709, MR2824469} respectively. For \textbf{Largely-open Nozzle Problem}, \cite{MR3148105} showed well-posedness of the subsonic flow by the variational method. Furthermore, for the incompressible cases, \cite{MR1280178} and \cite{MR4093720} proved the well-posedness of \textbf{Airfoil Problem} and \textbf{Largely-open Nozzle Problem}, respectively. For other studies on the subsonic flow with nonzero vorticity, one can refer \cite{MR3914482,MR3805847, MR3196988,MR2607929} and references therein.

Besides the well-posedness, it is natural to investigate the fine properties of the subsonic flows. One typical research field is the asymptotic behavior of the subsonic flows in the unbounded domains. The asymptotic properties of the subsonic flows act a significant role in the various physical settings such as the force and moment formula of the dynamics. Mathematically, it is also related to the uniqueness theory or the Liouville type theorem in the unbounded domains.

For \textbf{Airfoil Problem} in 2 dimension, the convergence rate of velocity fields is $|\bx|^{-1}$ which was also proved by Finn and Gilbarg in \cite{MR0086556}. Later on, for the $n$ dimensional case $(n\geq 3)$, Finn and Gilbarg \cite{MR0092912}
 obtained that velocity tends to a constant state at far fields with $|\bx|^{-\lambda+\epsilon}$ for any prescribed $\epsilon>0$, where $\lambda=\min{\{(n-2), 2\sqrt{n-1}\}}$, by estimating Dirichlet integral of each velocity component. The convergence rates are improved to $|\bx|^{2-n+\epsilon}$ by Payne and Weinberger in \cite{MR92913}. On the other hand, as the direct consequence of variational method, in the $n$ dimensional situations $(n\geq3)$, Dong and Ou\cite{MR1211737}  showed that the convergence rates of the velocity at far fields were $|\bx|^{-n/2}$. Similarly, for \textbf{Largely-open Nozzle Problem}, Liu and Yuan also showed that the convergence rates of the velocity at far fields were $|\bx|^{-n/2}$ in \cite{MR3148105}. For subsonic flow in the infinite long nozzle,  as the nozzle tends to the cylinder at infinite, the convergence rates of velocity was investigated in \cite{MR4128450}. Later, \cite{MR4536587} gave the convergence rates of velocity at far fields as the boundary of the nozzle goes to flat even when the forces do not admit
convergence rates at far fields, which revealed the influence of  the external forces on the convergence rates. For other studies on the subsonic flow with nonzero vorticity, one can refer \cite{MR3196988, MR4112818}.


In this paper, we will improve the order of convergence rates of the velocity of \textbf{Airfoil Far Field Flow} and \textbf{Largely-open Nozzle Far Field Flow}
from \cite{MR1211737, MR0092912, MR3148105, MR92913} 
to $|\bx|^{-n+1}$.  Comparing with \cite{MR0092912, MR92913}, we consider the potential function instead of each velocity component. By the primitive convergence rates, the potential function vanishes at infinity. And,  the key observation is that the coefficients, of which the potential function satisfies the quasi-linear elliptic equation, tend to limit states with certain convergence rates.  Using this fact and the elliptic property, we can choose a compared function to derive the estimates of the potential by the maximum principle, which leads to the higher convergence rates by the weighted Schauder estimate on the gradients of potential. This is established in section 2.  Unlike \cite{MR1211737,  MR3148105}, we concentrate the convergence rates of \textbf{Airfoil Far Field Flow} and \textbf{Largely-open Nozzle Far Field Flow}, do not rely on the well-posedness of \textbf{Airfoil Problem} and \textbf{Largely-open Nozzle Problem}. Furthermore, in section 3, we construct the examples to show the optimality of our convergence rates, and show the expansion of the incompressible  \textbf{Airfoil Problem} solution at infinity, which indicates the higher convergence rates.

\section{Asymptotic Behavior}
Before stating the main theorem of this paper, we introduce some notations. For any $0<R_1<R_2$, denote$$\Omega_{R_1}=\Omega\cap\{|\bx|\geq R_1\}\qaq\Omega_{R_1,R_2}=\Omega\cap\{R_1\leq|\bx|\leq R_2\}$$
are the simply connected part of $\Omega$. With abuse of notations, $\Omega$ can be regarded as $\Omega^{\mathcal B} $ or $\Omega^{\mathcal N}$ in the {\bf airfoil problem} or {\bf Largely-open nozzle problem}. And $B_{R_1}(\bx_0)\subset\mathbb{R}^n$ is the ball with radius $R_1$ centered at $\bx_0$.

The main Theorem can be stated as follows.
\begin{theorem}\label{mainthm} For $n\geq3$, suppose $\bu$ is the subsonic \textbf{Airfoil Far Field Flow} or \textbf{Largely-open Nozzle Far Field Flow}, there exists a positive constant $\sigma$ for the respective $\bu^\infty$ such that
	\begin{equation}\label{first_decay}
		|\bu(\bx)-\bu^\infty|\leq C|\bx|^{-1-\sigma}\quad\text{in}\quad\Omega_{R},\end{equation}
	while $C$ and $R$ are the fixed constants. Then, for some $R'\geq R$,
	\begin{equation}\label{main_estimate}
		|\bu(\bx)-\bu^\infty|\leq C_0|\bx|^{-n+1}\quad\text{when}\quad \bx\in\Omega_{R'},
	\end{equation}
	where $C_0$ is a uniform constant depending on $\Omega_{R'}$.
\end{theorem}

\begin{remark}
	The existence of $\bu$ satisfying \eqref{first_decay} could be found in \cite{MR1211737,MR3148105} for \textbf{Airfoil Far Field Flow} and \textbf{Largely-open Nozzle Far Field Flow}, respectively.
\end{remark}
\begin{remark}
	For $1+\sigma<n-1$, the estimate \eqref{main_estimate} indicates a better convergence rate than the one from  \eqref{first_decay}; for \cite{MR0092912, MR92913},  $\sigma=n-3-\epsilon$;
	 for \cite{MR1211737,MR3148105},  $\sigma=n/2-1$. 
\end{remark}

\begin{remark}  
	For $n=2$, the convergence rate of velocity fields is $|\bx|^{-1}$\cite{MR0086556}, which meets \eqref{main_estimate}, then  we need only consider the case for $n\geq3$.
\end{remark}

\subsection{Potential function} According to \eqref{irrotational}, in the simply connected set $\Omega_{R}$,
we can introduce the velocity potential $\varphi$ such that
$
	\nabla \varphi=\bu,
$
while the slip boundary condition \eqref{boundarycondition1} and \eqref{boundarycondition2} becomes
$
\nabla \varphi\cdot \bn=0.
$

 In virtue of Bernoulli's law $(\ref{eulereq})_3$, the density $\rho$ can be written as $\rho(|\nabla\varphi|^2)$. Hence, the mass conservation $(\ref{eulereq})_1$ can be reduced to the single quasi-linear equation
\begin{equation*}\label{divergence_form}
	\Div(\rho(|\nabla\varphi|^2)\nabla\varphi)=0.
\end{equation*}
It is easy to see that its non-divergence form is
\begin{equation}\label{non-divergence_form}
	\sum_{i,j=1}^n a_{ij}(\nabla\varphi)\partial_{ij}\varphi=0,
\end{equation}
where \begin{equation}\label{aijdefintion}
a_{ij}(\nabla\varphi)=\rho(|\nabla\varphi|^2)\delta_{ij}+2\rho'(|\nabla\varphi|^2)\partial_i\varphi \partial_j\varphi.
\end{equation}
Due to the regularity of $\rho(|\mathbf{p}|^2)$, for $\mathbf{p}$, $\mathbf{p'}\in\mathbb{R}^n$ satisfying $|\mathbf{p}|<q_{cr}$ and $|\mathbf{p'}|<q_{cr}$, one has
\begin{equation}\label{aijregularity}
	|a_{ij}({\mathbf{ p}})-a_{ij}({\mathbf {p}'})|\leq C |\mathbf{p}-\mathbf{p}'|,
\end{equation}
while $C$ depends on $\max\{|\mathbf{p}|, |\mathbf{ p}'|\}$. Next, we use $a_{ij}=a_{ij}(\nabla\varphi)$ unless otherwise specified. It is easy to see $(a_{ij})$ is a symmetric matrix.
For the uniformly subsonic flow, there exist two positive constants $\lambda$ and $\Lambda$ such that, for ${\bf \xi}\in\mathbb{R}^n$
$$
	\lambda|\xi|^2\leq \sum_{i, j=1}^na_{ij}\xi_i\xi_j\leq\Lambda|\xi|^2.
$$

Similarly, one could introduce the potential for the velocity at infinity, as $\varphi^\infty=\bu^\infty\cdot\bx$ with  $\nabla\varphi^\infty=\bu^\infty$.  For \textbf{Airfoil Far Field Flow}, $\varphi^\infty=q_1x_1$; and for \textbf{Largely-open Nozzle Far Field Flow},  $\varphi^{\infty}=0$. Since $\varphi^\infty$ is a linear function, it also holds \eqref{non-divergence_form}. Let the corresponding coefficients be
$$a_{ij}^\infty:=a_{ij}(\nabla\varphi^\infty)=a_{ij}(\bu^\infty).
$$
It follows from  \eqref{first_decay} and \eqref{aijregularity} that
\begin{equation}\label{aijproperty}
	|a_{ij}-a_{ij}^\infty|\leq C |\nabla\varphi-\nabla\varphi^\infty|\leq C|x|^{-1-\sigma}\quad \text{in $\Omega_R$}.
\end{equation}

 Next, instead of considering the difference between $\bu$ and $\bu^\infty$, we consider the difference of potential function $\Phi:=\varphi-\varphi^\infty$, with $\nabla \Phi=\bu-\bu^\infty$. Since $\varphi^\infty$ is a linear function, $\Phi$ also satisfies
 \begin{equation}\label{Phi}
 \sum_{i, j=1}^n a_{ij}\partial_{ij}\Phi=0,
 \end{equation}
 while $a_{ij}$ are defined as \eqref{aijdefintion} with property \eqref{aijproperty}. Then, the decay condition \eqref{first_decay} equals to
 \begin{equation}\label{lemma1estimate}
 	|\nabla\Phi|\leq C |\bx|^{-1-\sigma}.
 \end{equation}
To consider the infinity states of $\Phi$, one could introduce the Kevin transformation: $\by:=\frac{1}{|\bx|^2}\bx$, which maps $\Omega_R$ to $D'\subset B_{1/R}({\bf0})$. For $\by\in D'$, according to \eqref{lemma1estimate}, the respective function  $\tilde{\Phi}(\by):=\Phi(\frac{1}{|\by|^2}\by)$ satisfies
\begin{equation}\label{lemma1estimateY}
	|\nabla_{\by}\tilde{\Phi}(\by)|\leq\bigg|\frac{\partial \bx}{\partial\by}\nabla_{\bx}\Phi\bigg|\leq C|\by|^{-2}|\bx|^{-1-\sigma}\leq |\by|^{-1+\sigma}.
\end{equation}
 Then, we could conclude that there exists a unique constant $\Phi^\infty$ such that
	\begin{equation*}
	\lim\limits_{|\bx|\rightarrow+\infty\atop \bx\in\Omega_{R}}\Phi(\bx)=\Phi^\infty.
\end{equation*}
Without loss of generality, we assume $\Phi^\infty=0$, otherwise one could redefine $\Phi$ by $\Phi-\Phi^\infty$. Then there is a uniform constant $M_0$ such that
\begin{equation*}
	|\Phi(\bx)|\leq M_0 \quad\text{for } \bx\in\Omega_{R}.
\end{equation*}

Next, we will start to construct the comparison function concerning $\Phi$.
  Denote $(A_{ij})$ is the inverse matrix of $(a_{ij}^\infty)$, which also is a symmetric matrix, and $$Q(\bx):=\left(\sum_{i, j=1}^n A_{ij}x_ix_j\right)^{1/2}.$$
Then there exists a constant $C$, which only depends on $\lambda$ and $\Lambda$, such that
\begin{equation*}
\frac{1}{C}|\bx|\leq Q(\bx)\leq C|\bx|.
\end{equation*}
The following Lemma is about the estimates of $Q$ which is convenient to carry out the maximum principle.
\begin{lemma}\label{lemma2}
	 For any $0<\beta<1$, denote $\psi=Q^{2-n}-Q^{2-n-\beta}$. Supppse $a_{ij}$ satisfies \eqref{aijproperty}, then it holds that
	\begin{equation}\label{Lpsi}
\sum_{i, j=1}^n a_{ij}\partial_{ij}\psi=-\beta(n-2+\beta)Q^{-n-\beta}+O(Q^{-n-1-\sigma}).
	\end{equation}
\end{lemma}
\begin{proof}
	 For any $l>0$, the straightforward computations yield that
	\begin{equation*}
	\frac{\partial Q^{-l}}{\partial x_i}=-lQ^{-l-2}\sum_{k=1}^{n}A_{ik}x_k
	\end{equation*}
	and
	\begin{equation*}
	\frac{\partial^2 Q^{-l}}{\partial x_i\partial x_j}=l(l+2)Q^{-l-4}\bigg(\sum_{s=1}^n A_{js}x_s\bigg)\bigg(\sum_{k=1}^n A_{ik}x_k\bigg)-lA_{ij}Q^{-l-2}.
	\end{equation*}
Then,
\begin{eqnarray*}
\sum_{i, j=1}^n a_{ij}^\infty \frac{\partial^2 Q^{-l}}{\partial x_i\partial x_j}&=&lQ^{-l-2}\left((l+2)Q^{-2}\sum_{i, j=1}^n a_{ij}\left(\sum_{s=1}^n A_{js}x_s\right)\left(\sum_{k=1}^n A_{ik}x_k\right)-\sum_{i, j=1}^nA_{ij}a_{ij}\right)\nonumber\\
&=&lQ^{-l-2}\left((l+2)Q^{-2}\sum_{j=1}^n x_j\left(\sum_{k=1}^n A_{ik}x_k\right)-n\right)\nonumber\\
&=&l(l+2-n)Q^{-l-2},\nonumber
\end{eqnarray*}
which leads to
\begin{equation}\label{eq4}
\sum_{i, j=1}^n a_{ij}^\infty\frac{\partial^2Q^{2-n}}{\partial x_i\partial x_j}=0\quad \mbox{and}\quad \sum_{i, j=1}^n a_{ij}^\infty\frac{\partial^2Q^{2-n-\beta}}{\partial x_i\partial x_j}=\beta(n-2+\beta)Q^{-n-\beta}.	
\end{equation}
Combining with \eqref{aijproperty} and \eqref{eq4} yields
\begin{equation*}
	\begin{split}
\sum_{i, j=1}^n a_{ij}\partial_{ij}\psi&=\sum_{i, j=1}^n (a_{ij}-a_{ij}^{\infty})\partial_{ij}\psi+ \sum_{i, j=1}^n a_{ij}^{\infty}\partial_{ij}\psi\\
&=(a_{ij}-a_{ij}^{\infty})\bigg(\frac{\partial^2(Q^{2-n}-Q^{2-n-\beta})}{ \partial  x_i \partial x_j}\bigg)-\beta(n-2+\beta)Q^{-n-\beta}\\
&=	-\beta(n-2+\beta)Q^{-n-\beta}+O(Q^{-n-1-\sigma}).
\end{split}\end{equation*}
\end{proof}

Next, we will prove the Theorem \ref{mainthm} by \textbf{Airfoil Far Field Flow} and \textbf{Largely-open Nozzle Far Field Flow}, respectively.

\subsection{Airfoil Far Field Flow.}\label{subsection2.2} Here, we introduce
$$h_{\pm}(\bx):=\pm\frac{1}{C_1}\Phi(\bx)-\psi(\bx)$$
to estimate $|\Phi|$ via the comparison principle, while  $C_1>0$ will de determined later. Due to \eqref{Lpsi}, there exists a $R'$ large enough, such that
\begin{equation*}
\sum_{i, j=1}^n a_{ij}\partial_{ij}h_{\pm}(\bx)\geq \beta(n-2+\beta)Q^{-n-\beta}+O(Q^{-n-1-\sigma})\geq 0\quad\text{in } \Omega_{R'},
\end{equation*}
which implies $h_{\pm}(\bx)$ could not obtain the maximum in the interior of $\Omega_{R'}$. On the boundary of $\Omega_{R'}$, which is $\{|\bx|=R'\}$, one could have $h_{\pm}(\bx)\leq 0$, by choosing $C_1$ large enough. Furthermore, $\lim_{|\bx|\rightarrow \infty} h_{\pm}(\bx)=\lim_{|\bx|\rightarrow \infty}\left(\pm\frac{1}{C_1}\Phi(\bx)-\psi(\bx)\right)=0$. Then, we could conclude that $h_{\pm}(\bx)\leq 0$ in $\Omega_{R'}$. Then
\begin{equation*}
|\Phi(\bx)|\leq\frac{1}{C_1}\left(Q^{2-n}-Q^{2-n-\beta}\right)\leq C|\bx|^{2-n}.
\end{equation*}
Evidently, when $|\bx|$ large, one has
\begin{equation}\label{Phiestimate}
|\Phi|_{0;\Omega}^{(n-2)}\leq C,
\end{equation}
while the definition of $|\Phi|_{2,\tau;\Omega}^{(n-2)}$ is given in \cite[Section 6.1]{Gilbarg2001Trudinger}. Applying \cite[Lemma 6.20]{Gilbarg2001Trudinger} to the equation \eqref{Phi} yields
\begin{equation}\label{schauder}
|\Phi|_{2,\tau;\Omega}^{(n-2)}\leq C|\Phi|_{0;\Omega}^{(n-2)},
\end{equation}
where $0<\tau<1$ is a constant. It follows from \eqref{Phiestimate} and \eqref{schauder} that
\begin{equation*}
|\nabla\Phi|\leq C|\bx|^{1-n}.
\end{equation*}
This completes the proof of  \eqref{main_estimate} for the {\bf Airfoil Far Field Flow}.

\subsection{Largely-open Nozzle Far Field Flow.} The new effect of this case is the boundary nozzle, which is written as $\partial\mathcal C_{\pm, R}:=\{|x|\geq R\}\cap\partial\mathcal C_\pm$. Then $\nabla\Phi\cdot{\bf n}=\nabla\varphi\cdot{\bf n}=0$ on $\partial\mathcal C_{\pm, R}$.

In this case, the velocity at infinity is ${\bf{u^\infty}}=(0,0,\cdots,0)$.
Therefore,\begin{equation*}
	a_{ij}^\infty=\bar\rho\delta_{ij},
\end{equation*}
where $\bar{\rho}=h^{-1}(B)$ by  the Bernoulli's law $(\ref{eulereq})_3$.

Without loss of generality, we consider the $\mathcal C_+$, and $\mathcal C_-$ part could be handled similarly. After the coordinate transformation, one may assume $A_+=(0, 0,\cdots,0)$. We could introduce $Q$, $\psi$ and $h_\pm(\bx)$ as {\bf Airfoil Far Field Flow}.   Furthermore, on $\partial\mathcal C_{+, R}$, the directly calculations give ${\nabla Q}\cdot{\bf n}=0$. This implies: ${\nabla\psi}\cdot{\bf n}=0$ on $\partial\mathcal C_{+, R}$.

 On $\partial\mathcal C_{+, R'/2}$, $\nabla h_{\pm}(\bx)\cdot {\bf n}=0$. By the Hopf Lemma, $h_{\pm}(\bx)$ could obtain the maximum point at $\partial\mathcal C_{+, R'/2}$. On $\partial B_{+, R'/2}:=\{|\bx|=R'/2\}\cap\overline{\mathcal C_+}$, one could have $h_{\pm}\bx)\leq 0$, by choosing $C_1$ large enough. With the aid of  $\lim_{|\bx|\rightarrow \infty} h_{\pm}(\bx)=\lim_{|\bx|\rightarrow \infty}\left(\pm\frac{1}{C_1}\Phi(\bx)-{\psi}(\bx)\right)=0$. Then, we could conclude that $h_{\pm}(\bx)\leq 0$ in $\Omega_{R'/2}$. Therefore,
\begin{equation}\label{nozzlephi_L_infinity}
	|\Phi(\bx)|\leq\frac{1}{C_1}\left({Q}^{2-n}-{Q}^{2-n-\beta}\right)\leq C|\bx|^{2-n}.
\end{equation}

For any $R_1$ and $R_2$ satisfying $0<R_1<R_2$, denote
$$\mathcal C_{+}^{R_1}=\{|\bx|\geq R_1\}\cap\mathcal C_+\qaq\mathcal C_{+}^{R_1,R_2}=\{R_1\leq|\bx|\leq R_2\}\cap\mathcal C_+.$$

Finally, we use the Schauder estimate and \eqref{nozzlephi_L_infinity} to show that the convergence rates of $|\nabla\Phi(\bx)|$ is $|\bx|^{-n+1}$. For $\bx\in \mathcal C_+^{{T}/{2},{4T}}$, write $\Psi(\bz):=\Phi(T\bz)$ with $\bz=\frac{1}{T}\bx$. Then
 \begin{equation*}
	\sum_{i, j=1}^n a_{ij}(T\bz)\partial_{z_iz_j}\Psi=0.
\end{equation*}
It follows from the Schauder estimates that
\begin{equation*}
\sup\limits_{\bz\in\mathcal C_+^{1,2}}|\nabla_\bz\Psi(\bz)|\leq C\sup\limits_{\bz\in\mathcal C_+^{1/2,4}}|\Psi|.
\end{equation*}
This implies
\begin{equation}\label{CT}
T\sup\limits_{\bx\in\mathcal C_+^{T,2T}}|\nabla\Phi(\bx)|\leq C\sup\limits_{\bx\in\mathcal C_+^{T/2,4T}}|\Phi(\bx)|\leq
C|T|^{2-n},\end{equation}
where \eqref{nozzlephi_L_infinity} has been used. Note that $\mathcal C_+=\cup_{i=0}^\infty\mathcal C_+^{2^iR',2^{i+1}R'}$. For any $\bx\in \Omega_{R'}$, there is fixed $i\geq 0$ such that for $\bx\in\mathcal C_+^{2^iR',2^{i+1}R'}$. Applying \eqref{CT} yields that
\begin{equation*}
|\nabla\Phi(\bx)|\leq C|2^iR'|^{1-n}\leq C |\bx|^{1-n},\end{equation*}
which completes the proof for the {\bf Largely-open Nozzle Far field flow}. 
\subsection{Incompressible Case}
In this section, we consider the incompressible case, then \eqref{eulereq} becomes 
\begin{equation}\label{inceulereq}
	\begin{cases}
		\Div {\bf u} = 0,   \\
		{\rm curl}\  {\bf u}=0,
		\\
		\frac{1}{2}|\bu|^2+p\equiv B,
	\end{cases}
\end{equation}
where velocity $\bu$ and pressure $p$ are decoupled. 
Similar to the compressible case, one could define \textbf{Airfoil Problem}, \textbf{Largely-open Nozzle Problem}, \textbf{Airfoil Far Field Flow}, and \textbf{Largely-open Nozzle Far Field Flow} for the incompressible case. 

With the aid of the potential function $\varphi$ satisfying $\bu=\nabla\varphi$ can be governed as
\begin{equation*}
\Delta\varphi=0
\end{equation*}
with $a_{ij}=\delta_{ij}$. Similar to Theorem \ref{mainthm} we could have:
\begin{theorem}\label{mainthm1} For $n\geq3$, suppose $\bu$ is the incompressible \textbf{Airfoil Far Field Flow} or \textbf{Largely-open Nozzle Far Field Flow}, there exists a positive constant $\sigma$ for the respective $\bu^\infty$ such that
	\begin{equation}\label{first_decay1}
		|\bu(\bx)-\bu^\infty|\leq C|\bx|^{-1-\sigma}\quad\text{in}\quad\Omega_{R},\end{equation}
	while $C$ and $R$ are the fixed constants. Then, for some $R'\geq R$,
	\begin{equation}\label{main_estimate1}
		|\bu(\bx)-\bu^\infty|\leq C_0|\bx|^{-n+1}\quad\text{when}\quad \bx\in\Omega_{R'},
	\end{equation}
	where $C_0$ is a uniform constant depending on $\Omega_{R'}$.
\end{theorem}
\begin{remark}
	The existence of $\bu$ satisfying \eqref{first_decay} could be found in \cite{MR1280178, MR4093720} for \textbf{Airfoil Far Field Flow} and \textbf{Largely-open Nozzle Far Field Flow}, respectively.
\end{remark}

\section{Optimality}\label{optimality}
In this section, we proceed with the study of the optimal convergence rates of the velocity at far fields.

\subsection{\bf Largely-open Nozzle Far Field Flow}\label{sec3-1} Let $\bu$ satisfy \eqref{eulereq} with the slip boundary condition \eqref{boundarycondition2} and mass flux condition \eqref{flux}, and the respective potential is $\varphi$. Denote $$S_R=\mathcal C_+\cap\partial B_{R}(A_+).$$
Then $\Sigma_+=S_1$ as $R=1$. By \eqref{flux}, one has
\begin{equation}\label{optimaleq5}
	\int_{S_R}\rho(|\nabla\varphi|^2)\nabla\varphi\cdot {\mathbf l}ds=m>0\quad\text{for any}~R\geq1
\end{equation}
with $\mathbf l=({x_1/R,x_2/R,\cdots,x_n/R})$ be the unit  normal of $S_R$ pointed the downstream and $m$ is the mass flux.
Define $$u_0(r_+)=\frac{m}{|\Sigma_+|}r_+^{1-n}$$
with $r_+=|\bx-A_+|$.
Therefore, one has
\begin{equation}\label{optimaleq6}
	\int_{S_R}u_0(r_+)ds=m\quad\text{for any}~R\geq1.
\end{equation}
Combining \eqref{optimaleq5} and \eqref{optimaleq6} yields that
\begin{equation*}
	\int_{S_R}\rho(|\nabla\varphi|^2)\nabla\varphi\cdot {\mathbf l}-u_0(r)ds=0.
\end{equation*}
Then for each $R$ large enough, the mean value theorem gives that there is a constant $\bar\bx_{R}\in S_R$ such that
$$\rho(|\nabla\varphi(\bar\bx_{R})|^2)\nabla\varphi(\bar\bx_{R})\cdot \mathbf{l}(\bar \bx_{R})=u_0(|\bar\bx_{R}|).$$
Note that $\rho$ is bounded, one may assume that $\rho\leq\Lambda$. Then
\begin{equation*}
	\Lambda|\nabla\varphi(\bar\bx_{R})|\geq\rho(|\nabla\varphi(\bar\bx_{R})|^2)\nabla\varphi(\bar\bx_{R})\cdot \mathbf{l}(\bar \bx_{R})=u_0(R)=\frac{m}{|\Sigma_+|}R^{1-n}=\frac{m}{|\Sigma_+|}|\bar\bx_{R}|^{1-n}.
\end{equation*}
This shows the optimality of Theorem \ref{mainthm} for  compressible {\bf Largely-open Nozzle Far Field Flow}, which is similar to Theorem \ref{mainthm1} for incompressible case.
\subsection{Airfoil Far Field Flow}  Similar to {\bf Largely-open nozzle Far Field Flow}, one has
\begin{equation*}
	\int_{\partial B_R({\bf0})}\rho(|\nabla\varphi|^2)\nabla\varphi\cdot\bn ds=m
\end{equation*}
with $\mathcal B\subset B_R({\bf0})$ and $m$ is a constant.
If $m\not=0$, one may introduce the function $u_0(\bx)=\frac{m}{|\partial B_1({\bf0})|}|\bx|^{1-n}$, then
\begin{equation*}
	\int_{\partial B_R({\bf0})}\rho(|\nabla\varphi|^2)\nabla\varphi\cdot {\mathbf n}-u_0(\bx)ds=0.
\end{equation*}
As in Section \ref{sec3-1}, there is a point $\tilde\bx_R\in\partial B_R(\bf0)$ such that
\begin{equation*}
\Lambda|\nabla\varphi(\tilde\bx_R)|\geq\frac{m}{|\partial B_1({\bf0})|}R^{1-n}=\frac{m}{|\partial B_1({\bf0})|}|\tilde\bx_R|^{1-n}.
\end{equation*}
This shows the optimality of Theorem \ref{mainthm} for  compressible {\bf Airfoil Far Field Flow}, which is similar to Theorem \ref{mainthm1} for incompressible case.

Unfortunately, for the {\bf Airfoil problem}, the boundary condition \eqref{boundarycondition1} leads to $m=0$. In this case $u_0=0$, then this approach fail to show the optimality. In the next part, we will show the rates could be improved in the incompressible case. 
\subsection{Incompressible Airfoil Problem}\label{incompressible_case}
In this section, we shall show the expansion of incompressible airfoil flow at infinity. 

\begin{theorem}\label{thm2}
	For $\bu$ is the incompressible airfoil flow  satisfying \eqref{first_decay1}, then the corresponding potential function satisfies:
	\begin{equation*}
		\varphi(\bx)=q^\infty_1 x_1+\sum_{i=1}^n G_i\partial_{x_i}|\bx|^{2-n}+\sum_{i, j=1}^n G_{ij}\partial_{x_ix_j}|\bx|^{2-n}+O(|\bx|^{-n-1}),
	\end{equation*}
	where $G_i$ and $G_{ij}$ only depend on the value of $\varphi$ on $\partial\mathcal{B}$. 
\end{theorem}

\begin{proof}
The potential function $\varphi$ satisfies
\begin{equation}\label{incvarphiequation}
	\begin{cases}
		\Delta\varphi=0~ &\text{in}~\Omega^{\mathcal{B}},\\
			\nabla\varphi\cdot{\bf n}=0& \text{on}~\partial \mathcal B,\\
			\nabla\varphi=(q_1^\infty,0, \cdots,0) &\text{as}~ |\bx|\rightarrow\infty.
		\end{cases}
\end{equation}
From the  Section \ref{subsection2.2}, $\Phi(\bx)=\varphi(\bx)-q_1^\infty x_1+c_0$ with $c_0$ be the constant such that: 
\begin{equation} \label{Phi0infty0}
	\Phi(\bx)\rightarrow0~\text{as}~|\bx|\rightarrow \infty.
\end{equation} 
And, \eqref{incvarphiequation} becomes
\begin{equation}\label{incPhiequation}
	\begin{cases}
		\Delta_\bx\Phi=0~ &\text{in}~\Omega^{\mathcal B},\\
		\nabla_\bx\Phi\cdot \bn=-q_1^\infty n_1&\text{on}~\partial\mathcal B,\\
		\nabla_\bx\Phi={\bf 0} &\text{as}~ |\bx|\rightarrow\infty.
	\end{cases}
\end{equation}
 Integrating \eqref{incPhiequation} on $\partial\mathcal B$ leads to:
\begin{equation}\label{Phiboundaryint}
	\int_{\partial\mathcal{B}}\nabla\Phi\cdot \bn ds=0.
\end{equation}
For $r$ large enough, $\mathcal B\subset B_r(\mathbf0)$. By Kelvin transform  $\by=\frac{1}{|\bx|}\bx$ as $\mathbb{R}^n \setminus B_r(\mathbf0) \rightarrow B_{\frac{1}{r}}(\mathbf0)\setminus \{\mathbf0\}$, 
\begin{equation}\label{eq1}
	\bar\Phi(\by)=|\by|^{-n+2}\Phi\bigg(\frac{1}{|\by|^2}\by\bigg).
\end{equation}
From \cite[Theorem 4.13]{Gilbarg2001Trudinger}, in $B_{\frac{1}{r}}(\mathbf0)\setminus \{\mathbf0\}$,  $\Delta_\by\bar\Phi=0$. And, from \eqref{Phi0infty0}, as $|\by|\rightarrow0$,
\begin{equation*}
|\by|^{n-2}\bar\Phi(\by)=\Phi\bigg(\frac{1}{|\by|^2}\by\bigg)\rightarrow 0.
\end{equation*} 
Then, by \cite[Theorem 1.28]{MR2777537}, $\{\mathbf0\}$ is a removable singularity for $\bar\Phi$. Then, $\bar\Phi$ is harmonic in $B_{\frac{1}{r}}(\mathbf0)$, which implies analytic. Then, for $R'\geq R$, when $|\by|<\frac{1}{R'}$,
$\bar\Phi(\by)=\sum_{k=0}^{\infty}\sum_{|\alpha|=k}\frac{\partial^\alpha_y\bar\Phi(\mathbf0)}{\alpha!}\by^\alpha$, where $\alpha$ is the multi-index. Then, we have, for $|\bx|>R'$
\begin{equation}\label{preanalytic}
\Phi(\bx)=|\bx|^{-n+2}\bar\Phi(\frac{1}{|\bx|^2}\bx)=\sum_{k=0}^{\infty}\sum_{|\alpha|=k}\frac{\partial^\alpha_y\bar\Phi(\mathbf0)}{\alpha!}|\bx|^{-n+2-2|\alpha|}\bx^\alpha.
\end{equation}
Next, we will compute the coefficients of the expansion.
For any $\bx_0\in\Omega^{\mathcal B}$,  applying the Green's  second identity\cite[Equation 2.11]{Gilbarg2001Trudinger} in $B_R(\bx)\setminus(B_\epsilon(\bx)\cup\mathcal{B})$ with $\Phi$ and fundamental solution ${|\bz-\bx|^{-n+2}}$ yields that
\begin{equation*}
	\Phi(\bx)=-\frac{1}{n(n-2)\alpha(n)}\int_{\partial\mathcal{B}}(\Phi(\bz)\nabla(|\bz-\bx|^{-n+2})+|\bz-\bx|^{-n+2}\nabla\Phi(\bz))\cdot \bn ds(\bz).
\end{equation*}
Note that 
\begin{equation*}
|\bx-\bz|^{-n+2}=|\bx|^{-n+2}-\sum_{i=1}^n\partial_{x_i}|\bx|^{-n+2}z_i+\frac{1}{2}\sum_{i, j=1}^n\partial_{x_ix_j}|\bx|^{-n+2}z_iz_j+\sum_{k=3}^{\infty}\sum_{|\alpha|=k}\frac{\partial^\alpha|\bx|^{-n+2}}{\alpha!}\bz^\alpha.
\end{equation*}
Noticing for each fixed $|\alpha|=k$, $|\bx|^{-n+2-2|\alpha|}\bx^\alpha$ and $\partial^\alpha|\bx|^{-n+2}$ are one-to-one, \eqref{preanalytic} and the uniqueness of coefficients leads to 
\begin{equation*}
	\Phi(\bx)=G|\bx|^{2-n}+\sum_{i=1}^nG_i\partial_{x_i}|\bx|^{2-n}+\sum_{i, j=1}^nG_{ij}\partial_{x_ix_j}|\bx|^{2-n}+O(|\bx|^{-n-1})
\end{equation*}
with
\begin{equation*}
	\begin{cases}
		G=-\frac{1}{n(n-2)\alpha(n)}\int_{\partial\mathcal{B}}\nabla\Phi\cdot \bn ds,\\
		G_i=\frac{1}{n(n-2)\alpha(n)}\int_{\partial\mathcal{B}}\bigg(x_i\nabla\Phi\cdot \bn+n_i\Phi\bigg)ds,\\
		G_{ij}=\frac{1}{n(n-2)\alpha(n)}\int_{\partial\mathcal{B}}\bigg(-\frac{1}{2}x_ix_j\nabla\Phi\cdot \bn+x_in_j\Phi \bigg)ds.
	\end{cases}
\end{equation*}
And, $G=0$ by \eqref{Phiboundaryint}. $\Phi$ is the linear combination of the derivatives of fundamental solution.
\end{proof}
\begin{remark}
	As a direct consequence, one could have: For the incompressible airfoil flow $\bu$,
	\begin{equation}
		|\bu(\bx)-\bu^\infty|\leq C_0|\bx|^{-n} \quad\text{for } \bx\in\Omega_{R'}
	\end{equation}
	where $C_0$ and $R'$ are uniform constant depending on $\Omega^{\mathcal B}$.
	Now, the interesting questions are: whether the above estimate is optimal and whether it stands for the compressible case.
\end{remark}

\medskip
\textbf{Acknowledgement.}
The research was partially supported by NSFC grants 11971307 and 12061080.
The authors would like to thank Professor Chunjing Xie for  helpful discussions.

\end{document}